\documentclass[11pt,reqno,a4paper]{amsart}

\usepackage{amssymb}
\usepackage{amsmath}
\usepackage{amsthm}
\usepackage{mathrsfs}
\usepackage{color}
\usepackage{enumerate}
\usepackage{multicol}
\usepackage{hyperref}
\usepackage{pdfsync}

\newtheorem{thm}{Theorem}[section]
\newtheorem{pro}[thm]{Proposition}
\newtheorem{lem}[thm]{Lemma}
\newtheorem{cor}[thm]{Corollary}
\newtheorem{res}[thm]{Result}

\theoremstyle{remark}
\newtheorem{dfn}[thm]{Definition}

\newtheorem{examp}{Example}
\newenvironment{exa}[1]
  {\renewcommand\theexamp{#1}\examp}
  {\endexamp}

\renewcommand{\geq}{\geqslant}
\renewcommand{\leq}{\leqslant}

\newcommand{\floor}[1]{\left\lfloor #1 \right\rfloor}

\def\bS{\mathbb{S}}

\def\sgp{semigroup}

\DeclareMathOperator\PP{PPart}  

\DeclareRobustCommand{\stirling}{\genfrac\{\}{0pt}{}}
\DeclareRobustCommand{\comb}{\genfrac(){0pt}{}}

\begin{document}


\title[Nilpotent semigroups of index three]{Semirigidity and the enumeration of \\ nilpotent semigroups of index three} 


\author{IGOR DOLINKA}

\address{Department of Mathematics and Informatics, University of Novi Sad, Trg Dositeja Obra\-do\-vi\-\'ca 4,
21101 Novi Sad, Serbia}

\email{dockie@dmi.uns.ac.rs}

\author{D. G. FITZGERALD}

\address{School of Natural Sciences, University of Tasmania, Private Bag 37, nipaluna/Hobart 7001, Australia}

\email{d.fitzgerald.utas@icloud.com}

\author{JAMES D. MITCHELL}

\address{School of Mathematics and Statistics, University of St Andrews, St Andrews KY16 9SS, Scotland, United Kingdom}

\email{jdm3@st-and.ac.uk}

\thanks{This research was initiated during the second named author's visit to the University of Novi Sad in May 2018. The research of the first named author 
is supported by the Personal Grant F-121 ``Problems of combinatorial semigroup and group theory'' of the Serbian Academy of Sciences and Arts. 
The second named author thanks the University of Tasmania for assistance with travel and UniSuper Ltd for an indexed pension.}






\begin{abstract}
There is strong evidence for the belief that `almost all' finite semigroups, whether we consider multiplication operations on a fixed set or their isomorphism classes, 
are nilpotent of index $3$ ($3$-nilpotent for short). The only known method for counting all semigroups of given order is exhaustive testing, but formul\ae~ exist 
for the numbers of $3$-nilpotent ones, and it is also known that `almost all' of these are rigid (have only trivial automorphism).  
 
Here we express the number of distinct 3-nilpotent semigroup operations on a fixed set of cardinality $n$ as a sum of Stirling numbers, and provide a new expression 
for the number of isomorphism classes of 3-nilpotent semigroups of cardinality $n$. We introduce a notion of semirigidity for semigroups (as a generalization of rigidity) 
and find computationally tractable formul\ae~ giving an upper bound for the number of pairwise non-isomorphic semirigid 3-nilpotent semigroups, and thus an improved 
lower bound for the number of all 3-nilpotent semigroups up to isomorphism. Analogous formul\ae~ are also developed for isomorphism classes such as commutative and 
self-dual semigroups, and for equivalence classes (isomorphic or anti-isomorphic). The method relies on an application of the theory of orbit counting in permutation group actions.
Our main results are accompanied by tables containing values of these numbers and bounds up to $n=10$ with computations carried out in GAP (but perfectly feasible well beyond this value of $n$).
\end{abstract}


\maketitle

\vspace{-3mm}


\section{Introduction and Preliminaries}

Let $\mathscr{V}_{k}$ denote the class of semigroups $S$ satisfying $S^{k}=0$; it is a variety. Let $\mathscr{N}_{k}=\mathscr{V}_{k} \setminus \mathscr{V}_{k-1}$; $\mathscr{N}_{3}$ is the class of \emph{3-nilpotent} (or class 2) semigroups (in the strict sense); it is not a variety. Note that $\mathscr{V}_{3}$ has the law $xyz=0$, and $S \in \mathscr{N}_{3}$ if and only if $S \in \mathscr{V}_{3}$ and there exist $a,b \in S$ such that $ab \neq 0$.   
The most-studied properties of semigroups---distinctive features of ideals, location of idempotents, etc.---usually rule out 3-nilpotency.  Because  3-nilpotent semigroups are so prevalent, yet appear so featureless, they have been regarded as `junk' semigroups; and the main topic of study becomes that of how many there are.  

In 1976 Kleitman, Rothschild and Spencer \cite{KRS} gave an argument asserting that the proportion of 3-nilpotent semigroups, amongst all semigroups of order $n$, is asymptotically $1$ (although the agreement is poor below about $n=8$).  
Later opinion (e.g. \cite{JMS}) regards their argument as incomplete, and no satisfactory proof has been found.  Yet as a conjecture, the statement has considerable support.  What is certain is that the number of 3-nilpotent semigroups provides a lower bound for the number of  all \sgp s of given order, and so the construction of good fast methods for counting the 3-nilpotents remains relevant.     

The focus of \cite{KRS} is on semigroup operations on a fixed set of given size.  Different such semigroups may yet be isomorphic, or anti-isomorphic, and so count as the same in essence.  Thus it is of interest to ask analogous questions for isomorphism classes of semigroups, and for classes including isomorphic and anti-isomorphic copies too, which we call \emph{equivalence} classes.  We say of these enumerations that they are `up to identity',  `up to isomorphism', or `up to equivalence' respectively.   
 A  comprehensive  treatment of the matter by Distler and Mitchell \cite{DM} provides explicit (though complex) formul\ae~ for the numbers of 3-nilpotent semigroups up to identity, isomorphism, and equivalence, and for the corresponding numbers of commutative ones too. Their method used formul\ae~ from Chapter 10 of \cite{JMS}, with minor corrections.  

A partial exception to the lack of study of structural properties involves rigidity.  Recall that $S$ is said to be \emph{rigid} if its only automorphism is the identity map; otherwise it is \emph{flexible}.  The rigid semigroups of $\mathscr{N}_{3}$ form a significant subclass, studied in particular by P.\ A.\ Grillet in \cite{PAG2}, where he proved that `almost all' 3-nilpotent semigroups, up to identity and up to equivalence, are rigid. (His companion paper \cite{PAG1} is also noteworthy.) 
   
Let us write $t_{n}$ for the number of $\mathscr{N}_{3}$ semigroups of order $n$ (up to identity), and $r_{n}$ and $f_{n}$ for the numbers of rigid and flexible ones.  Let $\bar{t}_{n}$ and $\bar{\bar{t}}_{n}$ respectively be the numbers of isomorphism classes and equivalence classes of $\mathscr{N}_{3}$ semigroups of order $n$.  Further let $\bar{r}_{n}$ and $\bar{\bar{r}}_{n}$ stand for the numbers of rigid members of the respective classes, and in like manner, $\bar{f}_{n}$ and $\bar{\bar{f}}_{n}$ for the flexible ones.  Thus $t_{n} = r_{n} + f_{n}$, etc. 

\begin{res}[Grillet \cite{PAG2}, Theorem 4.5]  
$\lim_{n\rightarrow \infty}f_{n}/t_{n} = 0$, whence $r_{n}$  is an asymptotically good lower bound for $t_{n}$, in the sense that $\lim_{n\rightarrow \infty}r_{n}/t_{n} = 1$; moreover $\lim_{n\rightarrow \infty}\bar{\bar{f}}_{n}/\bar{\bar{t}}_{n} = 0$, though with a slower rate of convergence.
\end{res}

As a corollary, the analogous statement for numbers up to isomorphism is also true. 

\begin{cor}\label{corPAG} 
The quantity $\bar{r}_{n}$  is an asymptotically good lower bound for $\bar{t}_{n}$.    
\end{cor}

\begin{proof}
Since each equivalence class is made up of either one or two isomorphism classes, $\bar{f}_{n} < 2\bar{\bar{f}}_{n}$
and $\bar{\bar{t}}_{n} < \bar{t}_{n}$.  It follows that $\bar{f}_{n} / \bar{t}_{n} < 2\bar{\bar{f}}_{n}/\bar{\bar{t}}_{n}$ and hence that $\lim_{n\rightarrow \infty}\bar{f}_{n} / \bar{t}_{n} = 0$ and  $\lim_{n\rightarrow \infty}\bar{r}_{n} / \bar{t}_{n} = 1$.  
\end{proof} 

We continue with some general and structural remarks.  All semigroups in this paper will be finite and
have a zero element $0$ without further mention (so they form a variety of type $(2,0)$). 
Any semigroup has a chain of ideals $S \supseteq S^{2} \supseteq S^{3} \dots$; in a 3-nilpotent semigroup we have $S \supset S^{2} \neq 0$.  It is immediate that $S\setminus S^{2}$ ($= X$, say) is the unique minimum generating set for $S$ and that $S^{2}$ is the intersection of all the maximal proper ideals of $S$.   (Maximal ideals of $S$ have the form $S\setminus \{x\}$ for any $x \in S\setminus S^{2}$.) 

The following well-known construction is crucial.  
Given non-empty disjoint sets $X$ and $K$, an element $0 \notin X\cup K$, and a map $m:X \times X\rightarrow K\cup \{0\}$ such that $K\subseteq \text{im}(m)$, let $N(X,K,m)$ denote the magma with carrier set $X\cup K\cup \{0\}$ and multiplication $xy=m(x,y)$ if $x,y\in X$ and $xy=0$ otherwise.  Then $N(X,K,m)$ is in $\mathscr{N}_{3}$, as $x(yz)=0=(xy)z$.  Note that $m$ is just a minimal description of the Cayley table, and the condition $K\subseteq \text{im}(m)$ is the same as saying that the pointed map $m^{0}:(X \times X)\cup \{0\} \rightarrow K\cup \{0\}$ is surjective, or that $\text{im}(m)$ is either $K$ or $K\cup \{0\}$. 

Conversely, every 3-nilpotent semigroup $S$ can be written, up to isomorphism, in this form $S=N(X,K,m)$. Here $X$ is the
unique minimal generating set and $m$ is a function $X\times X\to K\cup\{0\}$ whose image is either $K$ or $K\cup\{0\}$; 
the carrier set of this semigroup is $X\cup K\cup\{0\}$. 

 Henceforth let $|S|=n, |X|=r$ and $|K|=k$, so $n-1=r+k$.  Upon fixing $X=\{1,\dots, r\}$ and   $K=\{r+1,\dots,r+k\}$, our first goal amounts to calculating the numbers of functions $m$ giving rise to semigroups $N(X,K,m)$, where $r+k+1=n$, and the numbers of their isomorphism classes.  Here $3$-nilpotent semigroups are represented as entirely combinatorial objects, as is implicit in \cite{KRS} and later literature; yet our approach will also involve understanding how properties of the functions $m$ are reflected in semigroup properties.


\section{Partial partitions and presentations}

 Let the partition of $X\times X$ induced by $m$ be denoted $\mathbf{P}_{m}=\{P_{0}, P_{1}, \dots, P_{k}\}$ where $P_{0}$ may be empty, but other $P_{j}$ are non-empty. Here $m(x,y) = r+j$ for $1\leq j\leq k$ when $(x,y)\in P_{j}$ and $m(x,y) = 0$ when $(x,y)\in P_{0}$.  It is useful to note that such partitions of a set $Y$ are equivalent to \emph{pointed partitions}  of the pointed set $Y\cup\{0\}$ (with $0\notin Y$). These consist of $k+1$ classes, in which $P_{0}$ is distinguished by containing the `point' $0$.   

Equally clearly, a function $m$ with the properties as described is also equivalent to a \emph{partial partition} of the set $X\times X$ 
into $k$ classes, that is a family of $k$ disjoint non-empty subsets of $X\times X$: for $1\leq j\leq k$ we set $(x,x')\in P_j$ if 
and only if $m(x,x')=r+j$.  Note that each such $P_j$ is non-empty, while for pairs $(x,x')$ not belonging to any of these 
sets we must have $m(x,x')=0$.  
We shall refer to these objects as either partial or pointed partitions, according to convenience.
 
Usually we shall write $m(x,y)$ as simply $xy$.  Then $\mathbf{P}$ specifies a presentation for $S$ within $\mathscr{V}_{3}$ in which the relations are 
 \begin{align*}
 &x_{1}y_{1}  =  x_{2}y_{2}, &&\text{   when   }(x_{1},y_{1}), (x_{2},y_{2}) \text{   are in the same block of   } \mathbf{P},  \\
 &xy    =  0,  &&\text{   when   }(x, y) \text{   is in no block of   } \mathbf{P}. 
\end{align*}    
So we say that $S$ is determined {\it up to presentation} by $\bf{P}$, and write $S = N(\mathbf{P})$, with the understanding that in $N(\mathbf{P})$ we define $xy$ to be $j$ when $(x,y)\in P_{j}$.  For a singleton block $\{(x,y)\}$ we must write $xy = xy$ in this convention, since unlisted pairs are assigned to $P_{0}$;  an alternative convention (used below) understands that unlisted pairs are singleton blocks, and the $P_{0}$ block is explicitly given.  
The number of 3-nilpotent semigroups of order $n$, up to presentation, is thus the number of such partitions $\mathbf{P}$ as $k$ runs from $1$ to its maximum.  

\begin{exa}{0} \label{ex0}
The free $3$-nilpotent semigroup $F_{r}$ of rank\footnote{We use rank in the usual two different senses, in the expectation that the context will make the intent clear: one for the cardinality of a generating set of a semigroup, the other for the cardinality of a partition (including that of the kernel of a transformation). } $r$ has $\mathbf{P}$ equal to the identity partition on $X\times X$; hence it has order $n=1 + r + r^{2}$. 
It has the empty presentation according to the second convention mentioned just above.
\end{exa} 	

In what follows, the Stirling number of second kind, the number of partitions of an $n$-set into $k$ parts, is denoted $\stirling{n}{k}$.  The corresponding number for partial partitions is given next.

\begin{lem}\label{nppart}
The number of partial partitions of an $n$-set into $k$ parts is $\stirling{n+1}{k+1}$. 
\end{lem}

\begin{proof}
 We have 
$\stirling{n}{k}$
such partitions that are actually total. On the other hand, a properly partial partition into $k$ classes is 
equivalent to a pointed partition into $k+1$ classes: we have 
$\stirling{n}{k+1}$
total partitions of the considered set into $k+1$ classes, which should be multiplied by $k+1$ choices for the 
class that will be pointed.  The required number is as claimed
$\stirling{n}{k} + (k+1)\stirling{n}{k+1} =  \stirling{n+1}{k+1}.$
\end{proof}

A presentation, together with the assignment of a zero and labels from $K$ to its classes, defines a $3$-nilpotent semigroup on the set $X\cup K\cup \{0\}$ of cardinality $r+k+1 = n$ uniquely, or as we say, \emph{up to identity}.  Thus we have:

\begin{lem}\label{1}
The number of $3$-nilpotent semigroups of order $n$, up to presentation, is $$\sum_{r=1}^{n-2}  \stirling{r^{2} + 1}{n-r},$$
and up to identity,  $$\sum_{r=1}^{n-2}\stirling{r^{2} + 1}{n-r} \frac{n!}{r!} = t_{n}.$$  
\end{lem}

\begin{proof}
 To construct a $3$-nilpotent semigroup with $r$ generators on a given $n$-set in a unique and exhaustive fashion, we first choose a set of generators ($\comb{n}{r}$ ways), then choose a presentation using them ($\stirling{r^{2} + 1}{n-r}$ ways)  and then apply $k+1 = n-r$ labels (including a zero) to the parts of the corresponding partition (the remaining elements of the semigroup)---this in $(k+1)! = (n-r)!$ ways. 
\end{proof}

\begin{table}[htb]
\begin{center}
{\footnotesize 
\begin{tabular}{l|r|r}
  $n$ & \# up to identity                  & \# up to presentation           \\ \hline
  3   & 6                                  & 1                               \\
  4   & 180                                & 15                              \\
  5   & 11~720                             & 536                             \\
  6   & 3~089~250                          & 74~875                          \\
  7   & 5~944~080~072                      & 55~046~362                      \\
  8   & 147~348~275~209~800                & 493~024~606~840                 \\
  9   & 38~430~603~831~264~883~632         & 75~797~430~892~164~879          \\
  10  & 90~116~197~775~746~464~859~791~750 & 120~455~109~059~841~172~414~778
 \end{tabular}
}
\end{center}
\vspace{2mm}
\caption{The number of 3-nilpotent semigroups up to identity, and up to presentation (Lemma \ref{1}).} 
\end{table}

In similar fashion we see that there are, up to presentation, 
$$B_{r^{2}+1}-1 = \sum_{k=1}^{r^{2}}  \stirling{r^{2} + 1}{k+1}$$  
$3$-nilpotent semigroups of rank $r$, where $B_{m}$ is the Bell number of index $m$.  Since there are infinitely many finite semigroups of rank $r$, the density of finite 3-nilpotent semigroups when stratified by rank is at the opposite extreme to that when stratified by order.  For a comment about the stratification by representation degree, see \cite{EENM}.

In Lemma \ref{1} the range of summation may be reduced to $r_{0}\leq r\leq n-2$, where $r_{0}$ is the least $r$ satisfying $r^{2}\geq k = n-r-1$.  
Then the expression for $t_{n}$ may be rewritten (using the inclusion-exclusion formula for a Stirling number) as 
\begin{equation}\label{eq:tn}
t_{n} = \sum_{r=r_{0}}^{n-2}\frac{n!}{r!}\frac{1}{(n-r)!}\sum_{j=0}^{n-r}(-1)^{n-r-j}\binom{n-r}{j}j^{r^{2}+1}.
\end{equation}
This is identical to the expression for $t_{n}$ in the Introduction to \cite{PAG2} and in Theorem 2.1 of \cite{DM}, as we now demonstrate.  The version from \cite{DM}, when cast in our present notation, with $k$ in place of $m-1$, is  
$$\sum_{k=1}^{b(n)}\binom{n}{k+1}(k+1)\sum_{i=0}^{k}(-1)^{i}\binom{k}{i}(k+1-i)^{r^{2}},$$
where $b(n) = n-1-r_{0}$.  
Now if we write $j=k+1-i, i=k+1-j$, and $n-r = k+1$, this becomes 
{\small $$
 \sum_{r=r_{0}}^{n-2}\binom{n}{n-r}\sum_{j=1}^{n-r}\frac{n-r}{j}(-1)^{n-r-j}\binom{n-r-1}{j-1}j^{r^{2}+1} 
 = \sum_{r=r_{0}}^{n-2}\binom{n}{n-r}\sum_{j=1}^{n-r}(-1)^{n-r-j}\binom{n-r}{j}j^{r^{2}+1}.
$$}
Taking into account the zero term for $j=0$, this indeed is (\ref{eq:tn}) above.  
The formula of Lemma \ref{1} provided a useful check on the implementation of the inclusion-exclusion formula used in \cite{DM}.  


\section{Commutativity and self-duality}   

A  3-nilpotent semigroup $S = N(X,K,m)$ is commutative when
$m(x,y) = m(y,x)$ (which we may describe as $m$ being symmetrical). 
 Define a subset $A$ of $X\times X$  to be 
\emph{symmetrical} if $(x,y)\in A$ implies $(y,x)\in A$, and \emph{skew} otherwise.  Say also that a (partial) partition is symmetrical if each of its blocks is symmetrical (the vacant or pointed block is also symmetrical in that case).  Then $S$ is commutative if and only if $m$ is symmetrical, if and only if its kernel $\mathbf{P}$ is symmetrical.  

To enumerate commutative $3$-nilpotent semigroups in the next Lemma, we replace $r^{2}$ ordered pairs of generators  by $\comb{r+1}{2}$ unordered pairs of generators. Other details follow the same reasoning as in 
Lemma \ref{1}. 
\begin{lem}\label{2}
The number of commutative $3$-nilpotent semigroups of order $n$, up to presentation, is $$\sum_{r=1}^{n-2}  \stirling{\frac{1}{2}r(r+1) + 1}{n-r},$$ and up to identity, $$\sum_{r=1}^{n-2}\stirling{\frac{1}{2}r(r+1) + 1}{n-r} \frac{n!}{r!}.$$
\end{lem}
\begin{proof}
A partial partition consisting of symmetric blocks is equivalent to a partial partition of a set consisting of the diagonal elements $(x,x)$ together with the unordered pairs $\{x,y\}$ with $x\neq y$.  (Substitute the unordered pair $\{x,y\}$ by the pairs $(x,y)$ and $(y,x)$.)  This set has $r + \comb{r}{2} = \frac{1}{2}r(r+1)$ elements, so the number of partial partitions of $X\times X$ of rank $k$ is $\stirling{\frac{1}{2}r(r+1)}{k+1}  = \stirling{\frac{1}{2}r(r+1)}{n-2}$.  Summation over $r$ gives the result as claimed.
\end{proof}

\begin{table}[hb]
\begin{center}
{\footnotesize
  \begin{tabular}{l|r|r}
  $n$ & commutative \# up to identity & commutative \# up to presentation \\ \hline
  3   & 6                             & 1                                 \\
  4   & 84                            & 7                                 \\
  5   & 1~620                         & 69                                \\
  6   & 67~170                        & 1~325                             \\
  7   & 7~655~424                     & 61~618                            \\
  8   & 2~762~847~752                 & 9~384~727                         \\
  9   & 3~177~531~099~864             & 5~668~560~557                     \\
  10  & 11~942~816~968~513~350        & 12~235~722~262~623
 \end{tabular}
 }
\end{center}
\vspace{2mm}
\caption{The number of \emph{commutative} 3-nilpotent semigroups up to identity and up to presentation (Lemma \ref{2}).}
\end{table}

\begin{exa}{1}\label{ex1} 
There is (up to presentation and indeed up to isomorphism) just one $3$-nilpotent semigroup of order $3$ and it is $F_{1}$ (free on one generator).  For order $4$ the only term in the sum in Lemma \ref{1} has $r = 2$, so there are $\stirling{5}{2}( = 15)$  $3$-nilpotent semigroups of order $4$ up to presentation.  They are listed below, according to the relations on generators $x, y$; note that to write e.g. $xy = yx = y^{2} = 0$ is meant to imply $x^{2} \neq 0$, etc.
{\small
\begin{multicols}{3}
\begin{enumerate} [(i)]
 \item $xy = yx = y^{2} = 0$;  
 \item $x^{2} = xy = yx =0$;
 \item $x^{2} = yx = y^{2} = 0$;
 \item $x^{2} = xy = y^{2} = 0$;
  
 \item $x^{2} = y^{2}, xy = yx = 0$;
 
 \item $x^{2} = xy, yx = y^{2} = 0$;
 \item $ x^{2} = xy = 0, yx = y^{2}$;
 \item $x^{2} = yx, xy = y^{2} = 0$;
 \item $x^{2} = yx = 0, xy = y^{2}$;
 
 \item $x^{2} = y^{2} = 0, xy = yx$;
 
 \item $xy = yx = y^{2},  x^{2} = 0$;
 \item $x^{2} = xy = yx,  y^{2} = 0$;
 \item $x^{2} = yx = y^{2}, xy = 0$;
 \item $x^{2} = xy = y^{2}, yx = 0$;

\item $x^{2} = xy = yx = y^{2}$ \\  (i.e., all $\neq 0$).
      
\end{enumerate}
\end{multicols}%
}
Twelve of these are isomorphic in pairs, determined by exchanging $x$ and $y$: (i) with (ii), (iii) with (iv), (vi) with (vii), etc.  The other three---(v), (x), (xv)---are fixed by the exchange; so there are nine isomorphism classes.  Five of these are commutative: (i)+(ii), (v), (x), (xi)+(xii), (xv).  Presentations (vi) and (viii) (or (vii) and (ix), etc.) provide examples of mutually dual non-isomorphic pairs, and this collapses their isomorphism classes into a single equivalence class ((vi) to (ix))---the only such collapse,  the other seven isomorphism classes being self-dual (i.e. they contain with each member its dual).  Thus there are eight equivalence classes. 

The sums in Lemma \ref{1} 
are respectively $\stirling{5}{2} = 15$ and 
$\stirling{2}{3} \frac{4!}{1!}  + \stirling{5}{2} \frac{4!}{2!}= 0 + 15\cdot 12 = 180$;
the sums in Lemma \ref{2} are $\stirling {4}{2} = 7$ and $\stirling {4}{2}\frac{4!}{2!} = 84$.
\end{exa}

It is germane to consider another relationship between semigroups. Given the semigroup $S = N(X,K,m)$, we define its \emph{dual} to be the semigroup $S^{\ast} = N(X,K,m^{\ast})$,  where $m^{\ast}(x,y) = m(y,x)$; moreover,   
 $S$ is said to be \emph{self-dual} if $S \cong S^{\ast}$.  Commutative  semigroups, of course, are self-dual.  To express these relationships by means of partitions, we may employ the \emph{twist} map
$\tau\colon X\times X \to X\times X$
 defined by $(x,y)\tau = (y,x)$.  The twist extends to subsets $P\subseteq X\times X$ and partitions $\mathbf{P}\in \PP(X\times X)$ by the definitions
 $P\tau = \{(x,x')\colon (x',x)\in P\}$  and 
 $\mathbf{P}\tau = \{P_{0},P_{1},\dots , P_{k}\}\tau = \{P_{0}\tau, P_{1}\tau,\dots ,P_{k}\tau\}$.  Then $S = N(\mathbf{P})$ is 
 commutative if and only if $P_{i}\tau = P_{i}$ for all $P_{i}\in \mathbf{P}$, and  $S$ is self-dual if and only if $\mathbf{P}\tau$ and $\mathbf{P}$ are isomorphic.  More generally, $S_{1}$ and $S_{2}$ are \emph{anti-isomorphic} if $S_{1}^{\ast} \cong S_{2}$, equivalently $\mathbf{P}_{1}\tau$ and $ \mathbf{P}_{2}$ are isomorphic.  We study isomorphisms more deeply in the next section, and return to anti-isomorphisms later on.

\begin{exa}{2}\label{ex2}
A non-commutative self-dual semigroup already arises in Example \ref{ex1}.  Presentations (iii) and (iv)
are both isomorphic and anti-isomorphic; likewise (xiii) and (xiv).  It is possible that the isomorphism involved be the identity; an example occurs 
when $r=2=k$.  With $X=\{x,y\}$ and $S$ presented by $x^{2} = y^{2} =0$, we have $S = \{x,y,xy,yx,0\} = S^{\ast}$ as sets, but with multiplications $xy$ and $x\ast y$ respectively, related by $x\ast y = yx$.  Thus if $\pi$ is the transposition $(x,y)$, $S\pi = S^{\ast}$, or $S\cong S^{\ast}$, while $\mathbf{P} = \{(x,y),(y,x)\} = \mathbf{P}\tau = \mathbf{P}^{\ast}$ (in obvious notation).
\end{exa}


\section{Actions and automorphisms; semirigidity}

As foreshadowed above, and to consider the question of enumeration up to isomorphism, etc., we use, and need to distinguish, at least three actions of the symmetric group $\bS_r$: first there is its defining or natural action on $X$, denoted $x\mapsto x\pi$ for  $\pi\in \bS_r$.  
Next there is the following \emph{cartesian} (square) action\footnote{This is usually called the \emph{diagonal} action, but we need to refer later to the diagonal of $X\times X$ and its associated act.} of $\bS_r$ on $X\times X$:
namely,  $(x,x')\pi=(x\pi,x'\pi)$.  For brevity we call this the \emph{c-action}.
The c-action naturally extends to an action of $\bS_r$ on the set $\PP_k(X\times X)$ of all partial partitions of
$X\times X$ into $k$ classes: namely, if $\mathbf{P}=\{P_1,\dots,P_k\}$, we define
$$
\mathbf{P}\pi = \{P_1\pi,\dots,P_k\pi\},
$$
which is easily seen to be again a partial partition of rank $k$.  We may refer to this action as the \emph{p-action}.  
The importance of these actions is explained in the following Lemma.
\begin{lem}\label{orb}
We have $N(X,K,m)\cong N(X,K,m')$ if and only if $\mathbf{P}_m$ and $\mathbf{P}_{m'}$ belong to the same orbit of the action 
of $\bS_r$ on $\PP_k(X\times X)$.
\end{lem}

\begin{proof}
Let $\phi:N(X,K,m)\to N(X,K,m')$ be an isomorphism. First of all, $\phi$ fixes the zero element $0$. Furthermore, as $X$ is
the unique minimal generating set of both semigroups involved, $\phi$ induces a permutation $\pi_\phi$ on $X$. This also implies 
that $K\phi=K$. Now clearly $m(x,x')\phi = m'(x\phi,x'\phi)$ for all $x,x'\in X$, implying that pairs $(x,x'),(y,y')$ belong 
to the same $\mathbf{P}_m$-class if and only if $(x\phi,x'\phi),(y\phi,y'\phi)$ belong  to the same $\mathbf{P}_{m'}$-class;
in other words, for all $1\leq j\leq k$, $P_j\phi$ is a $\mathbf{P}_{m'}$-class, and all $\mathbf{P}_{m'}$-classes arise in this way.
Hence, $\mathbf{P}_m\pi_\phi = \mathbf{P}_{m'}$, so $\mathbf{P}_m$ and $\mathbf{P}_{m'}$ belong to the same orbit of the 
 induced action on the set of partial partitions.

Conversely, assume there is a permutation $\pi$ of $X$ such that $\mathbf{P}_m\pi = \mathbf{P}_{m'}$. Then there is a permutation
$\widehat{\pi}$ of $\{1,\dots,k\}$ such that $P_j\pi=P'_{j\widehat{\pi}}$ for all $1\leq j\leq k$. To conclude, it is a matter of routine to
verify that the transformation $\phi$ of $X\cup K\cup\{0\}$ defined by
$$
a\phi = \begin{cases}
a\pi & \text{if } a\in X,\\
a\widehat{\pi} & \text{if }a\in K,\\
0 & \text{if } a=0
\end{cases}
$$
is in fact an isomorphism $N(X,K,m)\to N(X,K,m')$.
\end{proof}
There are several important corollaries to this result. 
\begin{cor}\label{auts}
\begin{enumerate}[(i)]
\item
This $\phi$ is an automorphism of $S = N(X,K,m)$ if and only if 
$\mathbf{P}_m\pi = \mathbf{P}_{m}$, which is to say that for each $i$, $P_{i}\pi = P_{j}$ for some $j$. 
\item
$S$ is self-dual if and only if $\mathbf{P}$ and $\mathbf{P}\tau$  lie in the same orbit of some $\pi$ in the p-action. 
\item If $S=N(\mathbf{P})$ is self-dual and $\pi\in\bS_{r}$, then $S\pi$ is self-dual. 
 \item
Isomorphism classes of 3-nilpotent semigroups of order $n$ and rank $r$ are in bijective correspondence with orbits of the induced action
of $\bS_r$ on $\PP_k(X\times X)$ with $r+k=n-1$.
\end{enumerate}
\end{cor}
 
In Section 1 we mentioned Grillet's study of the rigid members of $\mathscr{N}_{3}$  \cite{PAG2}.  
 It follows from (the proof of) Lemma \ref{orb} that rigidity is a property of the partition $\mathbf{P}$, and that $\mathbf{P}$ is {rigid} if and only if $\mathbf{P}\pi = \mathbf{P}$ implies $\pi = \text{id}$.  

Here we will be interested in a property weaker in general than rigidity: we say a semigroup $S$ is \emph{semirigid} if each automorphism of $S$ fixes every element of $S^{2}$.
(Of course the property is only distinctive if $S^{2}\neq S$.)

The definition is equivalent to the condition that each automorphism $\pi$ satisfies $(xy)\pi = xy$. 
So a  3-nilpotent semigroup
$S=N(\mathbf{P})$, or its partial partition  
$\mathbf{P} = \{P_{1},\dots,P_{k}\}$, is {semirigid} if and only if $\mathbf{P}\pi = \mathbf{P}$ implies that for each $i$, 
$P_{i}\pi = P_{i}$,   i.e., that each $\pi\in \bS_{r}$ fixes each block of  $\mathbf{P}$. 
Clearly a rigid $\mathbf{P}$ is semirigid, and it follows that  the number $s_{n}$ of semirigid $n$-element members of $\mathscr{N}_{3}$ is a lower bound for $t_{n}$ which improves on $r_{n}$. 
 
\begin{exa}{\ref{ex1} (continued)} \label{ex1-cont}
In the list for $n = 4$, it may be seen that cases (v), (x) and (xv) have the transposition $(x,y)$ as a non-identity automorphism, and so are not rigid.  The transposition however preserves each block,  so all three are semirigid.  The remaining cases are all rigid.
\end{exa}

\begin{lem}\label{subacts}
 Let $\mathbf{P}$ be [semi]rigid and $\pi\in \bS_{r}$.  Then $\mathbf{P}\pi$ is  [semi]rigid.
\end{lem}
\begin{proof}
Let $\sigma\in \bS_{r}$.  Then
$ \mathbf{P}\pi\sigma = \mathbf{P}\pi \Longrightarrow \mathbf{P}\pi\sigma\pi^{-1} = \mathbf{P}$.  So if 
$\mathbf{P}$ is rigid, $\pi\sigma\pi^{-1} = \text{id}$ and $\sigma = \text{id}$, whence  $\mathbf{P}\pi$ is rigid; whereas if $\mathbf{P}= \{P_{1},\dots, P_{k}\}$ is semirigid, $P_{i}\pi\sigma\pi^{-1} = P_{i}$ and $P_{i}\pi\sigma = P_{i}\pi$ for all $i$, whence $\mathbf{P}\pi$ is semirigid. 
\end{proof}
It is convenient to extend this nomenclature, and say that $\pi$ \emph{acts semirigidly} on $\mathbf{P}$ if 
$P_{i}\pi = P_{i}$ for each block $P_{i}$ of $\mathbf{P}$.  If $\mathbf{P}$ is indeed semirigid, then any $\pi\in\bS_{r}$ acts rigidly on it; but a $\pi$ may act semirigidly on a partition which is not semirigid.   At this stage we do not have a characterisation of [semi]rigid partitions; all we know is that, whatever they look like, they form, by Lemma \ref{subacts}, a subact. 


\section{{Orbit counting and partial partitions stabilised by $\pi$}}

From Corollary \ref{auts}(iv), we know that counting the distinct isomorphism classes requires counting the orbits of the action on partitions. 
For this we need the orbit-counting formula (see e.g.\ \cite[Theorem 2.2]{PJC}), associated with the names of Burnside and Frobenius, which declares that the number of orbits equals the average, over group elements, of the sizes of their sets of fixed points. 
In passing, we note the following result ({\it cf.} a bound in the proof of Theorem 4.5 in Grillet's \cite{PAG2}):
\begin{cor}\label{tn/n!}
The quantity $t_{n}/n! = \sum_{r=1}^{n-2}\frac{1}{r!}\stirling{r^{2} + 1}{n-r}$ is a lower bound for $\bar{t}_{n}$.
\end{cor}  
\begin{proof}
 Consider the action of the symmetric group $\bS_r$ on the presentations $\mathbf{P}$ with $r$ generators.  The identity of $\bS_r$ fixes all such, and so the average number of fixed partitions is at least $\frac{1}{r!}\stirling{r^{2} + 1}{n-r}$. This is therefore a lower bound for the number of isomorphism classes of rank $r$. 
\end{proof}

To proceed, we must study the partitions $\mathbf{P}$ fixed by a certain $\pi\in \bS_{r}$ in the p-action.  
We begin slightly more generally by describing the partitions fixed in any group action.  
Let $\bS$ be a group acting on a set $Y$ and let the action be extended as before to the [partial] partitions of $Y$.
Fix a $\pi \in \bS$;
we shall construct a partial partition of $Y$ stabilised by $\pi$, and later show all such are of this form.
\medskip

\noindent \textbf{Construction.}  Let us introduce a useful notion.    Given a set $G$ of cycles $\Gamma$ of the action of $\pi$, let $d$ be a common divisor of the lengths of those cycles.   
Let $E$ be a cross-section of $G$ (i.e., a set containing exactly one member of each $\Gamma \in G$), and define $F_{0} = \bigcup\{E\pi^{dq}\colon q\in \mathbb{N}\}$.  Consider the sets $F_{i} = F_{0}\pi^{i}$, for integers $i$ with $0\leq i < d$.
\medskip

\begin{lem}\label{simple}  
 $\mathbf{F} = \{F_{0}, \dots, F_{d-1}\}$ is a (proper) partition of \,$\bigcup G$ that is stabilised by $\pi$.  
\end{lem}
\begin{proof}
If $x\in F_{i}\cap F_{i'}$, so $x\pi^{qd+i} = x\pi^{q'd+i'}$, whence $x = x\pi^{(q-q')d+i-i'}$ and $d\mid (i-i') < d$; thus $i = i'$.  For any $x\in E$, $x\pi^{dq+i}\in F_{i}$ and so the cycle $\langle x\rangle \subseteq \bigcup F_{i}\subseteq \bigcup G$.  But $\bigcup G$ is the union of all such cycles $\langle x\rangle$, so $\bigcup F_{i} =  \bigcup G$.  Finally, $\mathbf{F}\pi = \mathbf{F}$ since $F_{d-1}\pi = F_{0}$.
\end{proof}
Consequently, $\mathbf{F}$ is a partial partition of $Y$ of rank $d$, which we call a \emph{frieze}.  
Given $G$, $\mathbf{F}$ is determined by $E$ and $d$, but not uniquely.   In passing we note that $\pi$ semirigidly fixes  $\mathbf{F}$ if and only if $d=1$.

Suppose $\lvert G\rvert = \lvert E \rvert = t$ and the cycle lengths are $\lambda_{e}$ for $1\leq e \leq t$.  We shall call  
$d$ the \emph{modulus}.  Immediately we have the next counting result.
\begin{lem}\label{nchoice}
The number of distinct friezes arising from $E$ and $d$ is $d^{t-1}$.
\end{lem}
\begin{proof}
Each block of $\mathbf{F}$ has $\lambda_{e}/d$ elements of $\Gamma_{e}$.  There are $\prod \lambda_{e}$ choices of cross-section $E$, of which $\prod (\lambda_{e}/d)$ give the same $F_{0}$.  Thus there are $d^{t}$ distinct choices of $F_{0}$, but $d$ of these give the same cycle 
$\mathbf{F}$.  
\end{proof}
We now extend Lemma \ref{simple}.  Let the cycles of $Y$ be partially partitioned into blocks $G_{\alpha}$ for $\alpha$ in some index set $A$; cross-sections $E_{\alpha}$ and moduli  $d_{\alpha}$ are indexed likewise.
 Construct as above a partition $\mathbf{F}_{\alpha}$ of $\bigcup G_{\alpha}$, and let $\mathbf{P}$ be a partial partition of $Y$ having the set of blocks  
$\bigcup_{\alpha\in A}\mathbf{F}_{\alpha}$ (recalling that the $\bigcup G_{\alpha}$ are pairwise disjoint).   
Then $\mathbf{P}$ is the disjoint union of the friezes determined by the $\mathbf{F}_{\alpha}$.
It is stabilised by $\pi$ and has $\sum_{\alpha} d_{\alpha}$ blocks. 
We write $\mathbf{P} = \mathbf{P}(\{E_{\alpha}\}, \{d_{\alpha}\})$ to show dependence on the choices determining $\mathbf{P}$.  
Given the $E_{\alpha}$ (which determines $G_{\alpha}$) and $d_{\alpha}$, the number of such distinct partitions  $\mathbf{P}$ is $\prod_{\alpha} d_{\alpha}^{t_{\alpha}-1}$ by Lemma \ref{nchoice}.  Now to the point: 

\begin{lem}\label{char}
Let $\mathbf{P}$ be a 
partial partition on $Y$ fixed by $\pi$.  Then $\mathbf{P}$ is  of the form  
$\mathbf{P} = \mathbf{P}(\{E_{\alpha}\}, \{d_{\alpha}\})$.
\end{lem} 
\begin{proof} 
By assumption, $\pi$ acts on the blocks of $\mathbf{P}$.  Let the cycles of this action be denoted by $\Delta_{\alpha}$, indexed by $\alpha\in \mathscr{A}$.   
Let $Q$ be a block of the cycle $\Delta_{\alpha}$,   and consider 
$R_{\alpha} = \bigcup \Delta_{\alpha} = \cup_{i\in \mathbb{N}}Q\pi^{i}$. 
Let $d_{\alpha}$ be the smallest $i > 0$ such that $Q\pi^{i} = Q$.  $R_{\alpha}$ is a union of cycles in the action on $Y$, $R_{\alpha} = \cup_{j\in J} \Gamma_{j}$ say, for $j$ in some implicit index set $J = J_{\alpha}$.  Now for each $j$, $Q\pi^{i}\cap \Gamma_{j}$ has cardinality independent of $i$, say $c_{j}$.  Then $c_{j}d_{\alpha} = \lambda_{j}$, 
 the length of $\Gamma_{j}$; thus $d_{\alpha}$ is a common divisor of the $\lambda_{j} (j\in J)$.  
 
For each $j\in J$, let $x_{j}\in Q\cap \Gamma_{j}$, and consider  $E_{\alpha} = \{x_{j}\colon j\in J\}$. Then $E_{\alpha}$ is a cross-section of the $\Gamma_{j}$, and $Q$ is a block of a frieze $ \mathbf{F}_{\alpha}$ (say), and thus of the disjoint union of the $ \mathbf{F}_{\alpha}$, to wit,  
$\mathbf{P}(\{E_{\alpha}\}, \{d_{\alpha}\})$.

Conversely, take a block $Q'$ of one of the friezes of $\mathbf{P}(\{E_{\alpha}\}, \{d_{\alpha}\})$ constructed as above.  By construction, $Q'$ is a block of one of the cycles of $\mathbf{P}$.  It now follows that $\mathbf{P} = \mathbf{P}(\{E_{\alpha}\}, \{d_{\alpha}\})$. 
\end{proof}
We may now apply Lemma \ref{char} to the action of $\bS_r$ on the set $\PP(X\times X)$ of all partial partitions of
$X\times X$. For brevity we call a cycle of the cartesian action a \emph{c-cycle} to distinguish it from a cycle in the defining action, which we may call an \emph{x-cycle} for emphasis. 
\begin{pro}\label{the list}
Let $\mathscr{G}$ be the set of partial partitions $\mathbf{G} = \{G_{\alpha}\}$  of the c-cycles of $\pi$, indexed by $\alpha\in \mathscr{A}$. (Our notation is abbreviated for simplicity: $\mathscr{A}$ depends on  $\mathbf{G}\in \mathscr{G}$.)
\begin{enumerate}[(i)]
\item The number of partial partitions of $X\times X$  fixed by $\pi\in \bS_{r}$ is  
$$
\sum_{\{G_{\alpha}\} \in \mathscr{G}}\prod_{\alpha}d_{\alpha}^{t_{\alpha}-1}.
$$
\item The number of partial partitions of $X\times X$ semirigidly fixed by $\pi\in \bS_{r}$ is  $$\sum_{\{G_{\alpha}\} \in \mathscr{G}}1 = | \mathscr{G}|.$$
\item The number of semirigid partial partitions of $X\times X$  fixed by $\pi\in \bS_{r}$ is bounded above by 
$ | \mathscr{G}|.$
\item The number of semirigid partial partitions of rank $k$ fixed by $\pi\in \bS_{r}$ is bounded above by the number of members of $\mathscr{G}$ having $k =\vert \mathscr{A}\vert$ blocks.
\end{enumerate}
\end{pro}

\begin{proof}
\begin{enumerate}[(i)]
\item To choose a partial partition $\mathbf{P}$, 
one must choose a partial partition $\mathbf{G}$ of the c-cycles of $\pi$ as in Lemma \ref{char}, having blocks and corresponding integers $d$  indexed by $\alpha\in\mathscr{A}$.  
By Lemma \ref{nchoice}, the number of friezes on block $\alpha$ is 
$d_{\alpha}^{t_{\alpha}-1}$,
and choices for these may be made independently.  Hence the number of such partitions for $\mathbf{G}$ is 
$\prod_{\alpha}d_{\alpha}^{t_{\alpha}-1}$.
Summing over the choices of $\mathbf{G}$ gives the result.
\item By the remark after Lemma \ref{simple}, if $\mathbf{P}$ is semirigidly fixed by $\pi$, we must have $d_{\alpha} =1$ for all $\alpha$, so $\prod_{\alpha}d_{\alpha}^{t_{\alpha}-1} = 1$ and (ii) follows.  In this case we can also see that each block $G_{\alpha}$ is a union of c-cycles.
\item Each semirigid partition is semirigidly fixed by $\pi$ (but not necessarily conversely), so the expression in (ii) is an upper bound.
\item By (ii), $k=\sum_{\alpha\in \mathscr{A}}1 = |\mathscr{A}|$, and the result follows. 
\end{enumerate}
\end{proof}

\begin{exa}{3}\label{ex3}
In Example \ref{ex1}, with $k = 1$, necessarily $|\mathscr{A}| = 1$ and $d = 1$.  Then partitions fixed by  $\pi$ are semirigidly fixed, and equality holds in Proposition \ref{the list}(iii).  
However with $r=k=2$ (so $n=5$), we have $d=2$ for one pair of c-cycles of the transposition $\pi$, namely $(xx,yy)$ and $(xy,yx)$, where we write $xy$ for the pair $(x,y)$ in the cause of brevity.  Then there are two possible inequivalent choices for $E$: $\{xx,xy\}$ and $\{xx,yx\}$.  These give respectively partitions 
$\mathbf{P}_{1} = \{\{xx,xy\},\{yy,yx\}\}$ and 
$\mathbf{P}_{2} = \{\{xx,yx\},\{yy,xy\}\}$.  The resulting $S$ are fixed by $\pi$, but not semirigidly, and are not included in the count of Proposition \ref{the list}(ii).   
\end{exa}

Doubtless the calculation of Proposition \ref{the list} is precisely that of the power-groups method used in Distler and Mitchell \cite{DM}.


\section{Bounding counts up to isomorphism}

To use these facts for enumeration of isomorphism classes, we may again apply the orbit-counting formula.  This yields the following proposition, where we continue to use the notation of Proposition \ref{the list}.
\begin{pro} \label{exact}
The number of orbits of the action of $\bS_r$ on $\PP_k(X\times X)$ is 
$$
\sum_{\pi\in \bS_{r}} \frac{1}{r!}\sum_{\{G_{\alpha}\} \in \mathscr{G}}\prod_{\alpha}d_{\alpha}^{t_{\alpha} - 1}
$$
subject to 
$\sum d_{\alpha} = k$.
\end{pro}

There are small examples of this calculation in Examples \ref{ex5} and \ref{ex6} below.  We remark that the dominant term in such a sum comes by choosing $d_{\alpha}=1$ for each $\alpha$, since this captures all the rigid cases, and $\lim_{n\rightarrow \infty}\bar{r}_{n} / \bar{t}_{n} = 1$ by Corollary \ref{corPAG}. 

To tease out the expression in Proposition \ref{exact} we must describe the c-cycles of a permutation $\pi$.  
 The following basic facts about c-cycles are straightforward to prove. 
\begin{lem}\label{key}
Let us assume that the permutation $\pi$ has cycles
$C_1,\dots,C_s$ for its natural action on $X$ ($|X|=r$), and let $|C_i|=\lambda_i$. 
\begin{enumerate}[(i)]
\item Each cycle of the cartesian action of $\pi$ on $X\times X$ is contained in $C_i\times C_j$ for some $1\leq i,j\leq s$.
\item In fact, each block $C_i\times C_j$ consists precisely of $(\lambda_i,\lambda_j)$ cycles (the g.c.d), 
each of length 
$[\lambda_i,\lambda_j]$ (the l.c.m).  
\end{enumerate}
\end{lem}

Let us first note that the products and summands in Proposition \ref{exact} are dependent only on the cycle type, i.e., conjugacy class of the particular $\pi\in \bS_{r}$, and thus on the integer partition associated with $\pi$, 
$$ 
\lambda = (\lambda_1^{\mu_1},\dots,\lambda_t^{\mu_t}),
$$
$\mu_{i}$ being the multiplicity of $\lambda_{i}$-cycles.  We write $\lambda \vdash r$ to indicate that $\lambda$ is an integer partition of $r$.  In concrete cases the notation $\lambda_1^{\mu_1}\dots\lambda_t^{\mu_t}$ for $\lambda$ is convenient.  We will also use the fact that the corresponding conjugacy class contains 
$$
\frac{r!}{\lambda_1^{\mu_1}\dots\lambda_t^{\mu_t}\mu_1!\dots \mu_t!}
$$
permutations.  Let us introduce the symbol  
$\beta_{d} = \beta_{d}(\lambda)$ for the number of c-cycles with length divisible by $d$, 
$$
\beta_{d}(\lambda) = \sum_{d\vert [\lambda_i,\lambda_j]} (\lambda_i,\lambda_j).
$$
(The  sum is taken over ordered pairs.)  In particular, write simply $\beta(\lambda)$ for 
$$
\beta_{1}(\lambda) = \sum_{1\leq i,j\leq s} (\lambda_i,\lambda_j),
$$
which is the total number of c-cycles of any $\pi$ with cycle type $\lambda$, and $\beta(\pi)$ for $\beta(\lambda)$ when $\pi$ has cycle type $\lambda$.  We note that for $r$ fixed, $\beta(\lambda)$ attains its maximum of $r^{2}$ at $\lambda = 1^{r}$.

Calculating with Proposition \ref{exact} looks difficult, so we proceed by obtaining approximate formul{\ae}  for the numbers of semirigid isomorphism classes (the by-far dominant term). 

\begin{lem}\label{tech}
The number of semirigid partial partitions of $X\times X$ of rank $k$  fixed by $\pi\in \bS_{r}$ is bounded above by 
$$
F(\pi) = \stirling{\beta(\pi)+1}{k+1}.
$$
\end{lem}

\begin{proof}
By Corollary \ref{the list}(iv), we may use the cardinality of $\PP_k(X\times X)$ and this is given by Lemmas \ref{nppart} and \ref{key}.  
\end{proof} 
Applying the orbit-counting formula and  taking into account that $k +1 = n-r$, we immediately obtain 

\begin{pro} \label{upbd}
The number of orbits of the action of $\bS_r$ on the semirigid members of $\PP_k(X\times X)$ is bounded above by
$$\frac{1}{r!} \sum_{\pi\in \bS_{r}}\stirling{\beta({\pi})+1}{n-r} .$$
\end{pro}

Note  the bound given here need not be an integer.    
For instance, the identity semirigidly fixes each $\mathbf{P}\in \PP_{k}(X\times X)$, and so  if $\mathbf{P}$ is not semirigid it contributes an excess of ${\stirling{r^{2}+1}{n-r}}/{r!}$ to the relevant orbit count---see Example \ref{ex4} below.  Nonetheless, the expression in Proposition \ref{upbd} is itself bounded above by the number of orbits of $\bS_r$ acting on \emph{all} members of $\PP_k(X\times X)$.  We may (slightly) refine a rational estimate of an integer quantity by replacing it by its integer part (floor function value) or its ceiling value, depending on whether we seek an upper bound or a lower bound.  In computations below, we carry fractions in intermediate steps, and round down at the end. 

Now a bound for the  number of non-isomorphic semirigid 3-nilpotent semigroups of order $n$ is obtained by summation of the terms
from the previous proposition with $r$ ranging over all possible values.  However, we would rather like to re-arrange this sum by grouping together the (equal) terms corresponding to
permutations with the same cycle type on $X$, that is, those belonging to the same conjugacy class. 
Thereby we obtain the following result.

\begin{thm} \label{count}
The number of non-isomorphic semirigid 3-nilpotent semigroups of cardinality $n$ is bounded above by
$$
\sum_{r=1}^{n-2} \sum_{\lambda\vdash r} \frac{\stirling{\beta(\lambda)+1}{n-r}}{\lambda_1^{\mu_1}\dots\lambda_t^{\mu_t}\mu_1!\dots \mu_t!} ;
$$
and this quantity is also a lower bound for the number of all non-isomorphic 3-nilpotent semigroups of cardinality $n$.
\end{thm}

\begin{table}[hbt]
\begin{center}
\footnotesize{
 \begin{tabular}{l|r|r|r|}
  $n$        & $a_n =$ \# s'rigid up to iso & $b_n =$ Theorem \ref{count}
  u.b. & $c_n = $\# 3-nilp. up to iso 
  \\ \hline
  3           & 1                          & 1
              & 1   \\%
  4           & 9                          & 9
              & 9    \\%
  5           & 114                        & 116
              & 118  \\%
  6           & 4~629                      & 4~650
              & 4~671  \\%
  7           & 1~198~759                  & 1~199~370
              & 1~199~989  \\%
  8           & -                          & 3~661~477~300
              & 3~661~522~792  \\%
  9           & -                          & 105 931 863 102 354
              & 105 931 872 028 455  \\%
  10          & -                          &  24 834 563 575 435 688 559 
	            & 24 834 563 582 168 716 305 \\
 \end{tabular}
 }
 \end{center}
\vspace{2mm}
\caption{The number of \emph{isomorphism classes} of 3-nilpotent semigroups: the actual number of semirigid ones, its upper bound from Theorem \ref{count},
  and the actual total from Smallsemi \cite{GAP}.}
\end{table}

It is convenient to use the notation $w(\lambda) = {\lambda_1^{\mu_1}\dots\lambda_t^{\mu_t}\mu_1!\dots \mu_t!}$ and refer to it as a \emph{weight} in sums such as the above.  

\begin{exa}{1 (redux)} \label{ex1-red}
Let $n = 4$. Then the outer sum in Theorem \ref{count} has just one term, with $r = 2$. The partitions $\lambda \vdash 2$ are $1+1 = 1^{2}$ and $2 = 2^{1}$, and give the corresponding values $\beta(1^{2}) = (1,1)+(1,1)+(1,1)+(1,1) = 4$ 
and $\beta(2^{1}) = 2\wedge 2 = 2$. Thus Theorem \ref{count} evaluates to ${\stirling{5}{2}} /{2} + {\stirling{3}{2}}/{2} =9$, as it should (\cite{DM}, Table 3 and references), since all these semigroups are semirigid.    
\end{exa}

\begin{exa}{3 (redux)} \label{ex3-red}
Let $n = 5$. There are contributions only from $r=2$ (when $k=2$) and $r = 3$ (when $k=1$). For $r=2$ the partitions $\lambda \vdash 2$ are $1+1 = 1^{2}$ and $2 = 2^{1}$ as before, and the inner sum of Theorem \ref{count} for $d=1$ evaluates to ${\stirling{5}{3}} /{2} + {\stirling{3}{3}}/{2} =\bf{13}$. 
For $d = 2$, we previously found the contribution $\bf{2}$, and $d> 2 = k$ does not contribute.  
For $r=3$, $d=1$ since $k=1$, and we have  $\beta(1^{3}) = 9$, $\beta(1^{1}2^{1}) = 5$ and $\beta(3^{1}) = 3$, with partial sum 
${\stirling{10}{2}}/{3!} + {\stirling{6}{2}}/{2} +{\stirling{4}{2}}/{3} = {511}/{6}+{31}/{2}+{7}/{2} =\bf{103}$.  
So the total number of isomorphism classes is $118$. 
\end{exa}

\begin{exa}{4} \label{ex4}
Let $n = 7$; with $|X| = r = 2$, so $k = 4 = r^{2}$, we have the free 3-nilpotent semigroup $F_{2}$.  Its partition $\mathbf{P}$ is the identity partition and it has $\bS_{2}$ for its automorphism group.  Though it is not semirigid, the identity map nonetheless acts semirigidly on $\mathbf{P}$ and it contributes a quantity $\frac{1}{2}$ to the bounding sum of Theorem \ref{count}, which is therefore non-integral. 
\end{exa}

\begin{exa}{5} \label{ex5} 
For $n = 6$, the formula of Theorem \ref{count} evaluates to $\bf{4650}$, and this includes all cases where $d=1$.  Cases with $d=2$ come from $r=2, \lambda = 2^{1}, \beta_{2} = 2$ and $r=3, \lambda = 1^{1}2^{1}, \beta_{2}=4$.  No other cases arise, since for $\lambda = 3^{1}$ a choice with $d=3$ contradicts $k=2$; and for $r=4$ we have $k=1$ and no frieze is possible.     
For $\lambda = 2^{1}, k = 3$, so of the two c-cycles, one must have $d=2$ and the other $d=1$.  This comes about in just 2 ways, and $w(2^{1})={2}$. This contributes $\bf{1}$ to the count of isomorphism classes.  For $\lambda = 1^{1}2^{1}, k = 2, w(1^{1}2^{1})={2}$, just one block of $t$ c-cycles makes a frieze with $d=2$ parts.  There are $\binom{\beta_{2}}{t}$ such blocks and each may be turned into a frieze in $2^{t-1}$ ways (Proposition \ref{the list}).  This gives  $\sum_{t=1}^{4}\binom{4}{t}2^{t-1}=\frac{1}{2}((2+1)^{4})-1 = 40$ partitions and an additional term of $\bf{20}$.  Thus we have a total of $4650+1+20 = \bf{4671}$ isomorphism classes, in agreement with \cite{DM}.
\end{exa}

The dominant term in the inner sum in Theorem \ref{count} is that for $\lambda = 1^{r}$.  Its value is ${\stirling{r^{2}+1}{n-r}}/{r!}$, and so the double sum is at least $\sum_{r=1}^{n-r}{\stirling{r^{2}+1}{n-r}}/{r!} = t_{n}/n!$, the quantity encountered in Corollary \ref{tn/n!}.

We may continue to compute analogous `correction' terms, associated with integer partitions.    
Thus, one way to proceed might be to list the correction terms according to the cycle type.  Our zeroth guess may be that cycle types $1^{i}2^{j}$ are major contributors, as they are least constrained.  So let us begin with a small extension of Example \ref{ex5}.  Let $\lambda = 1^{\mu_{1}}2^{\mu_{2}}$, so 
$r = \mu_{1} + 2\mu_{2}$, $\beta_{1} = (\mu_{1}+\mu_{2})^{2} + \mu_{2}^{2}$, $\beta_{1} - \beta_{2} = \mu_{1}^{2}$, and $\beta_{2} = 2\mu_{2}(\mu_{1}+\mu_{2})$; also $w(\lambda)=2^{\mu_{2}}\mu_{1}!\mu_{2}!$.  There is a partial correction term of rank $k$, made from any possible frieze of modulus $d=2$ on the $\beta_{2}$ c-cycles and, independently chosen, any partial partition of the remainder having rank $k - 2$: 
$$
\frac{\sum_{t=1}^{\beta_{2}}\binom{\beta_{2}}{t}2^{t-1}
\stirling{\beta_{1}-t+1}{k-2+1}}{2^{\mu_{2}}\mu_{1}!\mu_{2}!} .
$$
This expression simplifies in a couple of cases: if $k=2$, the Stirling number is $1$ and the summation term is
$
\sum_{t=1}^{\beta_{2}}\binom{\beta_{2}}{t}2^{t-1}
=\frac{1}{2}(3^{\beta_{2}}-1)
$. If $k=3$, we may use the identity 
$\stirling{\beta_{1}-t+1}{2} = 2^{\beta_{1}-t}-1$ to express the sum as
\begin{align*}
\sum_{t=1}^{\beta_{2}}\binom{\beta_{2}}{t}2^{t-1}(2^{\beta_{1}-t}-1) &=
\frac{1}{2}\sum_{t=1}^{\beta_{2}}\binom{\beta_{2}}{t}2^{\beta_{1}} - \frac{1}{2}\sum_{t=1}^{\beta_{2}}\binom{\beta_{2}}{t}2^{t}\\
&= \frac{1}{2}(2^{\beta_{1}+\beta_{2}} - 2^{\beta_{1}} - 3^{\beta_{2}} +1).
\end{align*}
 
\begin{exa}{6} \label{ex6}
We can apply these to find the remaining `non-semirigid' (i.e., $d>1$) terms for $n=7$.  We already had a term for $d=1$ (by GAP calculation, and verified by hand), $\bf{1199370.5}$; it (over-) estimates the semirigid isomorphism classes---including the fractional part noted in Example \ref{ex4}.
This already includes the case $k = 1$, since $\sum d_{\alpha} = k$.  
If $k = 2$, $r = 4$, and the relevant integer partitions are $\lambda = 1^{2}2^{1}, 2^{2}$, and $4^{1}$.  We take them in turn.

For $\lambda = 1^{2}2^{1}$, we have $\beta_{1} = 10$ and $\beta_{2} = 6$, so 
$\frac{1}{2}(3^{\beta_{2}}-1) = \frac{1}{2}(3^{6}-1) =\frac{1}{2}(728)$, while $w = w(\lambda) = 4$. This accounts for $\mathbf{91}$ isomorphism classes.
For $\lambda = 2^{2}$, $w = 8$, $\beta_{1} = \beta_{2} = 8$, and the partition count evaluates to $\frac{1}{2}(3^{8} - 1) = 3280$, giving an additional term of $3280/8 = \mathbf{410}$ classes.   For $\lambda = 4^{1}$, $\beta_{1} = \beta_{2} = w = 4$, and $\frac{1}{2}(3^{4} - 1) = 40$, giving a further $40/4 = \mathbf{10}$ classes.

Now we go to $k = 3$ and $r=3$; here $\lambda = 1^{1}2^{1}$ or $3^{1}$.  In the first case, we have $w = 2, \beta_{1} = 5$, and $\beta_{2} = 4$.  Our additional count of partitions is (as given above)  
$\frac{1}{2}(2^{9} - 2^{5} - 3^{4} +1) = 200$ and of classes, $\mathbf{100}$.  
The same reasoning as used above applied to $\lambda = 3^{1}$, $\beta_{1} = \beta_{3} = w = 3$, when a single frieze with $d = 3$ is available, gives a partition count of 
$\sum_{t=1}^{\beta_{3}}\binom{\beta_{3}}{t}3^{t-1} = 
3^{-1}((3+1)^{3} - 1) = 21$
and a class count of $21/w = \mathbf{7}$.

Lastly, if $k=4, r=2$, the only friezes feasible are with $\lvert\mathscr{A}\rvert = 2, d_{1} = d_{2} = 2, t_{1}=t_{2}= 1$ and $\lambda =2^{1}$.  Then $w(2^{1}) = 2$ and the number of partitions resulting is $d_{1}^{t_{1}-1}d_{2}^{t_{2}-1} = 2^{0}2^{0}= 1$;  the contribution to the class count is therefore $\mathbf{\frac{1}{2}}$.

In summary, we have found a further $91+410+10 + 100+7+0.5 = \mathbf{618.5}$ isomorphism classes.  Combined with the `semirigid' ones noted above, we have a total of $\bf{1199989}$ classes,  the correct value from \cite{DM}.  
\end{exa}


\section {Commutative classes}

To bound the number of semirigid commutative isomorphism classes, we need to determine when the partitions of Proposition \ref{upbd} are symmetric.  
Consider a certain permutation $\pi$ having cycle type $ \lambda $, as before. 
 Recall (from the proof of part (ii) of Lemma \ref{the list} that a partial  partition $\mathbf{P}$ is semirigidly fixed by $\pi$ if and only if each class is a
 union of c-cycles of $\pi$ acting on $X\times X$.  Moreover, 
 such a $\mathbf{P}$ is symmetric if and only if each class is a union of symmetric cycles and matched pairs of skew cycles. Observe next that the cartesian action has two invariant sets, the diagonal $\{(x,x)\colon x\in X\}$ and its complement; and that each cycle contained in the diagonal is symmetric. Others are identified with the aid of the next lemma. 
\begin{lem} \label{sym}
Let $\Gamma = \{(x,y)\pi^{i}\colon 1\leq i\leq s\}$ be a non-diagonal cycle of length $s$ in the cartesian action. 
Then $\Gamma$ is symmetric if and only if $s$ is even and $\Gamma = \{(x,x\pi^{\frac{s}{2}})\pi^{i}\colon 1\leq i\leq s\}$, for some $x$ in a cycle $C_{j}$ of length $s$ in the defining action of $\bS_{r}$ on $X$.

\end{lem}
\begin{proof}
Suppose $\Gamma$ symmetric, and $(x,y)\in \Gamma$. Then $(y,x) = (x\pi^{i},y\pi^{i}) $ for some $i$, whence $x=x\pi^{2i}$ for some $i$ assumed minimal, and so $s = 2i$ and $(x,y) = (x,x\pi^{i})$. Conversely, if $s$ is even, consider $(x, y) \in \Gamma $ with $x\in C_{j}$ as described, and $y = x\pi^{\frac{s}{2}}$; then $(y,x) = (x\pi^{\frac{s}{2}}, x\pi^{s}) = (x, y)\pi^{\frac{s}{2}} \in \Gamma $ and $\Gamma$ is symmetric.
\end{proof}

This establishes the following set-up.
A c-cycle $\Gamma$ contained in $C_{i}\times C_{j}$ where $i\neq j$ is \emph{ipso facto} skew and matched with $\Gamma \tau \subseteq C_{j}\times C_{i}$.  If $\lambda_{i} = \vert C_{i}\vert $ is odd, all non-diagonal c-cycles in $\Gamma\subseteq C_{i}\times C_{i}$ (there are $\lambda_{i}-1$ of them) are skew and occur in pairs.
 If $\lambda_{i} = \vert C_{i}\vert $ is even, the diagonal and one other c-cycle are symmetric and the rest skew and occur in pairs.
 
Thus in the cycle type $\lambda= (\lambda_1^{\mu_1},\dots,\lambda_t^{\mu_t})$ there are $\sum_{i} \mu_{i}$ diagonal cycles, while the number of symmetric cycles as described in Lemma \ref{sym} is the number of even-length cycles in $\pi$.  So if we define $e(n)$ to be 1 if $n$ is even and $0$ otherwise, there are $\sum_{i} \mu_{i}e(\lambda_{i})$ even-length cycles, and $\delta(\lambda) = \sum_{i} \mu_{i}(1 + e(\lambda_{i}))$ symmetric cycles altogether. The remaining cycles number $\beta(\lambda) - \delta(\lambda)$ and occur in matched pairs.   
(Note these counts are dependent not on $\pi$ but only on its cycle type.)

It follows that the set of building blocks for orbits for symmetric partial partitions semirigidly fixed by $\pi$ is of size $$\frac{1}{2}(\beta(\lambda) - \delta(\lambda)) + \delta(\lambda) = \frac{1}{2}(\beta(\lambda) + \delta(\lambda)).$$
For brevity, we denote this value by $\gamma(\lambda)$.

Arguing as in Theorem \ref{count}, this means that for commutative isomorphism classes, $\beta(\pi)$ should be replaced
by $\gamma(\pi)$, when we have 

\begin{thm}\label{comm}
The number of non-isomorphic semirigid commutative 3-nilpotent semigroups of cardinality $n$ is bounded above by
$$
\sum_{r=1}^{n-2} \sum_{\lambda\vdash r} \frac{\stirling{\gamma(\lambda)+1}{n-r}}{\lambda_1^{\mu_1}\dots\lambda_t^{\mu_t}\mu_1!\dots \mu_t!} .
$$
\end{thm}

\begin{exa}{1 and 3 (even more)}  \label{ex13-more}
First with $n=4$: for $\lambda = 1^{2}, \lambda_{1}=1, \mu_{1}=2$, and $\delta(1^{2})=2(1+0)=2$;  then $\gamma(1^{2})=\frac{1}{2} (\beta(1^{2}) + 2) = 3.$
For $\lambda = 2^{1},  \lambda_{1}=2, \mu_{1}=1$, and $\delta(2^{1})=1(1+1)=2$, whence $\gamma(2^{1})=\frac{1}{2}(2+2)=2$.  The formula from Theorem \ref{comm} evaluates to $\frac{1}{2}\stirling{4}{2} + \frac{1}{2}\stirling{3}{2} = 5$, as reported in  \cite{DM}, Table 6.   With $n=5$, the relevant $\lambda$ are $1^{2}, 2^{1}, 1^{3}, 1^{1}2^{1}$ and $3^{1}$.  The $\beta$ values as calculated in Example \ref{ex3} are $4,2,9,5,3$ respectively, with corresponding $\delta$ values $3,2,6,4,2$; so $\gamma = 3,2,3,3,1$.  The weights are $2,2,6,2,3$ so the estimate from Theorem \ref{comm} is $\stirling{4}{3}/2 + \stirling{3}{3}/2 + \stirling{7}{2}/6 + \stirling{5}{2}/2 + \stirling{3}{2}/3 = 22.5$, cf. the exact $23$ for all commutative 3-nilpotents from \cite{DM}. 
\end{exa}


\section{Counts up to equivalence and self-dual classes}

We may also consider two semigroups to be essentially the same if they are anti-isomor\-phic, that is, one is isomorphic to the dual of the other.  The transitive closure of this relation is the \emph{equivalence} relation, that is, being either isomorphic or anti-isomorphic, which has also been called `up to anti-isomorphism'.
The pairing implicit in the above discussion is induced by the twist 
$\tau\colon X\times X \to X\times X$ defined by $(x,y)\tau =(y,x)$, which extends to an action on the semirigid
partial partitions on $X\times X$: for (as is easily seen) $\mathbf{P}\tau$ is semirigid if and only if $\mathbf{P}$ is.  
It further extends to an action on the \emph{orbits} of those partitions under $\mathbb{S}_{r}$, equivalently, to isomorphism classes of the relevant semigroups.  To be specific, the isomorphism class of $S$ maps to the isomorphism class of $S^{*} = S\tau$.  
  
We shall use the following well-known fact, for which we present a proof (just because it is thematic).  
 \begin{lem}\label{folk}
 Let $\mathscr{X}$ be a set of isomorphism classes of semigroups, and $\mathscr{Y}$ the set of self-dual classes in $\mathscr{X}$.  Then the number of equivalence classes of these semigroups (i.e. up to anti-isomorphism) is $\frac{1}{2}(|\mathscr{X}|+|\mathscr{Y}|)$.
\end{lem}
  
\begin{proof}
The equivalence classes are precisely the orbits of the action of $\langle\tau\rangle =\{\text{id},\tau\}$.  Now $\text{id}$ fixes each element of $\mathscr{X}$ and $\tau$ fixes each element of $\mathscr{Y}$, so the claim follows from the orbit-counting lemma.
\end{proof}

Our strategy is to take $\mathscr{X}$ to consist of the orbits of semirigid 3-nilpotent semigroups $S=N(\mathbf{P})$, identified in Lemma \ref{tech} and counted (in the sense of finding an upper bound) in Proposition \ref{upbd}.  
Then the elements of $\mathscr{Y}$ are the orbits of semirigid partial partitions $\mathbf{P}$ satisfying 
 $\mathbf{P}\pi\tau = \mathbf{P}$ for some $\pi$, a sub-act since $\mathbf{P}\sigma\pi^{\sigma}\tau = \mathbf{P}\sigma$.  Our next task is to identify and count those partitions and their orbits.  

So let $\mathbf{P}$ be a partial partition on $X\times X$ such that $\mathbf{P}=\mathbf{P}\pi\tau$.  
Then $\mathbf{P}\pi^{2} = \mathbf{P}$, and since $\mathbf{P}$ is semirigid, we have $P_{i}\pi^{2} = P_{i}$ for each block $P_{i}$.  
Thus $\mathbf{P}$ is formed by partitioning the c-cycles of $\pi^{2}$, so that each block $P_{i}$ is a union of c-cycles of $\pi^{2}$.  We shall impose conditions on $\mathbf{P}$ so that $\mathbf{P} = \mathbf{P}\pi\tau$ (without necessarily forcing $P_{i} = P_{i}\pi\tau$).

It is convenient to denote the set of c-cycles of $\pi^{2}$ by $\mathscr{C}$.   
A previous remark shows that $\mathbf{P}$ induces a partial partition $\mathbf{P}^{\dag}$ on $\mathscr{C}$ such that if $(x,y)$ and $(x_{1},y_{1})$ belong to the same block of $\mathbf{P}$, then $(x,y)\langle\pi^{2}\rangle$ and $(x_{1},y_{1})\langle\pi^{2}\rangle$ belong to the same block of $\mathbf{P}^{\dag}$. Conversely, each block of $\mathbf{P}$ is the union of cycles $\Gamma$ in the same block of $\mathbf{P}^{\dag}$.  It is worth noting at the outset  that if $\Gamma\in\mathscr{C}$ then $\Gamma\pi ,\Gamma\tau\in\mathscr{C}$.

There are three kinds of cycles $\Gamma \in \mathscr{C}$: diagonal cycles, which need separate attention, those for which $\Gamma = \Gamma\pi\tau$, which we call \emph{singular}, and the rest which we call \emph{regular}.  The regular ones occur in distinct pairs $\Gamma, \Gamma\pi\tau$ which we call \emph{associates}.   
Singular cycles are rare: 
\begin{pro}\label{sing}
There is one singular cycle for each x-cycle of odd length, and it is diagonal; and there are two for each x-cycle of twice-odd length, and no others.
\end{pro}
\begin{proof}
 If $\Gamma = \Gamma\pi\tau$ and $(x,y)\in \Gamma$, then $(x,y)=(y\pi,x\pi)\pi^{2k}$ for some $k$, whence $x=y\pi^{2k+1}$ and $y=x\pi^{2k+1}$, so that $x=x\pi^{4k+2}$ and $4k
+2\equiv 0\pmod{\lambda_{i}}$.  This has 0,1, or 2 solutions for $k\bmod{\lambda_{i}}$ according as $(\lambda_{i},4)= 4,1\text{ or }2$.  Hence if $\lambda_{i}\equiv 0\pmod{4}$ there is no solution, if $\lambda_{i}\equiv 1,3\pmod{4}$ the unique solution is $k=(\lambda_{i}-1)/2$, and if $\lambda_{i}\equiv 2\pmod{4}$ the solutions are $k\in\{(\lambda_{i}-2)/4,(\lambda_{i}-2)/4+\lambda_{i}/2\}$. 

These solutions lead respectively to $2k+1=\lambda_{i}, \lambda_{i} /2$, and $3\lambda_{i} /2$. Thus either $\lambda_{i}\equiv 1\pmod{2}$, $x=y\pi^{\lambda_{i}}= y$, and $\Gamma$ is diagonal; or $\lambda_{i}\equiv 2\pmod{4}$ and $y=x\pi^{\lambda_{i}/2}=x\pi^{3\lambda_{i}/2}$. In this latter case, it is readily verified that $\Gamma=\{(x,x\pi^{\lambda_{i}/2})\}\langle\pi^{2}\rangle\in \mathscr{C}$ satisfies $\Gamma=\Gamma\pi\tau$, as does $\Gamma\pi$; thus $\Gamma$ and $\Gamma\pi$ are singular and distinct (and non-diagonal). 
\end{proof}

We can now characterise partitions $\mathbf{P}^{\dag}$ of $\mathscr{C}$ such that $\mathbf{P}=\mathbf{P}\pi\tau$ for a given $\pi$.  One should note that a $\mathbf{P}$ constructed as described below need not be semirigid, so the set of all such may properly contain $\mathscr{Y}$, and its cardinality is an upper bound for $|\mathscr{Y}|$ . 

\begin{thm}\label{thm:C}
Given $\pi\in\mathbb{S}_{r}$ and the collection $\mathscr{C}$ of all c-cycles of $\pi^{2}$, choose a subset (allowed to be empty) of associated pairs of regular c-cycles, and partition it completely so that no block contains both members of an associated pair.  Let this partition be denoted $\mathbf{Q}$ (empty if the chosen subset is empty).  Then choose a partial partition $\mathbf{R}$ of the remaining c-cycles, chosen so each block contains, along with a c-cycle $\Gamma$, its associated c-cycle $\Gamma\pi\tau$.  That is, each such block consists of singular cycles and \emph{unions of} associated pairs.

Then the union $\mathbf{P}^{\dag} = \mathbf{R}\cup\mathbf{Q}$ corresponds to a partial partition $\mathbf{P}$ of $X\times X$ which satisfies $\mathbf{P}\pi\tau = \mathbf{P}$.
Moreover, every semirigid partial partition $\mathbf{P}$ fixed by $\pi\tau$ is uniquely of this form. 

 \end{thm}
\begin{proof}
The first conclusion is manifest from the construction.  
For the converse, suppose $\mathbf{P}$ is semirigid and $\mathbf{P}\pi\tau = \mathbf{P}$.  For each block $P_{i}$ of $\mathbf{P}$, there holds $P_{i}\pi\tau = P_{j}$ for some $j$, and we can write $j = j(i)$.  This gives an involutory permutation $j$ of the blocks of $\mathbf{P}$, such that if the c-cycle $(x,y)\langle\pi^{2}\rangle$ lies in $P_{i}$, then $(y\pi,x\pi)\langle\pi^{2}\rangle$ lies in $P_{j(i)}$.  Note that this ensures that $P_{j(i)}$ is determined by $P_{i}$. 

Let $\mathscr{Q}$ be the set of $\Gamma\in\mathscr{C}$ such that $\Gamma$ and $\Gamma\pi\tau$ are in different blocks $P_{i}, P_{j}\, (i\neq j(i))$, and let $\mathscr{R} = \mathscr{C}\setminus \mathscr{Q}$.  Necessarily, a c-cycle in $\mathscr{Q}$ is regular and all singular c-cycles involved in $\mathbf{P}^{\dag}$ are in $\mathscr{R}$.
There is a natural matching on $\mathscr{Q}$ in which $\Gamma\subseteq P_{i}$ and its associate $\Gamma\pi\tau\subseteq P_{j(i)}$ are matched. Note that the restriction $\mathbf{Q}$ of $\mathbf{P}^{\dag}$ to $\mathscr{Q}$ is a total partition but that $\mathscr{Q}$ may be empty. 
On the other hand, if $\Gamma\in \mathscr{R}$ and $\Gamma \subseteq P_{i}$, then $\Gamma\pi\tau \subseteq P_{i}$ also.  Thus $\mathbf{P}$ has the form described.
\end{proof}
The description in Theorem \ref{thm:C} now allows us to enumerate the partitions fixed by $\pi\tau$.  First, we catalogue the c-cycles $\Gamma$ of $\pi^{2}$ in terms of the parameters of $\pi$, in which an x-cycle $C_{i} = x\langle \pi\rangle$ has length $\lambda_{i}$, etc.  We refer to the sets $C_{i}\times C_{j}\subseteq X\times X$ as \emph{panels}.  

We need some more notation.
Let $f_{i}=1$ if $\lambda_{i}\equiv2\pmod{4}$ and $f_{i}=0$ otherwise; let $g_{i} = 1$ if $\lambda_{i}\equiv0\pmod 4$ and $g_{i} = 0$ otherwise (so $e_{i} = f_{i}+g_{i}$); and let $e_{ij} = 1$ if one or both of $\lambda_{i},\lambda_{j}\equiv0\pmod{2}$ and $e_{ij}=0$ otherwise (so $e_{ij} = e_{i}+e_{j}-e_{i}e_{j}$). 

\begin{thm}\label{thm:sap}
The c-cycles of $X\times X$ required for Theorem \ref{thm:C} when $\pi$ has cycle type $\lambda$ consist of $\eta(\lambda)$ singular cycles and $\zeta(\lambda)$ associate pairs, where 
$$
\zeta(\lambda)  =  \frac{1}{2}\sum_{i}\{(1+e_{i})(\lambda_{i}+1) - 2(f_{i}+1)\} + \sum_{i<j}\{1+e_{ij}\}(\lambda_{i},\lambda_{j})
$$ 
and
$$
\eta(\lambda) =  \sum_{i}(1-e_{i}+2f_{i}).
$$
\end{thm}

\begin{proof}
Consider the panel $C_{i}\times C_{i}$. 
\begin{itemize}
\item  If  $\lambda_{i}\equiv 1 \pmod 2$, there is one singular diagonal cycle and $(\lambda_{i}-1)/2$ pairs of associates.  \item  If $\lambda_{i}\equiv 0 \pmod 4$, the c-cycles of $\pi^{2}$ are of length $\lambda_{i}/2$ and there are $2\lambda_{i}$ of them, occurring in $\lambda_{i}$ associate pairs $\{\Gamma, \Gamma\pi\tau\}$. 
\item If $\lambda_{i} \equiv 2\pmod 4$, there are, by Proposition \ref{sing}, two (off-diagonal) singular cycles, and $\lambda_{i}-1$ associate pairs (of which one makes up the diagonal).
\end{itemize}
Thus the number of singular cycles here is $(1-e_{i}) + 2f_{i}$, and summing over $i$ gives $\eta(\lambda)$.

Moreover, the number of associate pairs from the list above  is 
\begin{equation*}
 (1-e_{i})(\lambda_{i}-1)/2 + g_{i}\lambda_{i} + f_{i}(\lambda_{i}-1) 
= \frac{1}{2}(1 + e_{i})(\lambda_{i} + 1) - 1 - f _{i}.
\end{equation*}
Now consider a panel $C_{i}\times C_{j}$, with $i < j$.  For a c-cycle $\Gamma\subseteq C_{i}\times C_{j}$, $\Gamma \tau\pi\subseteq C_{j}\times C_{i}$, and so these form an associate pair.  
\begin{itemize}
 \item If $\lambda_{i}\equiv\lambda_{j}\equiv 0 \pmod 2$, the length of $\Gamma$ is $[\lambda_{i}/2,\lambda_{j}/2]$ and the number of such $\Gamma$ in $C_{i}\times C_{j}$ is $\lambda_{i}\lambda_{j} / [\lambda_{i}/2,\lambda_{j}/2] = 4(\lambda_{i}/2, \lambda_{j}/2) = 2(\lambda_{i},\lambda_{j})$.  We have a representative of each associate pair if we select $i< j$.
 \item If $\lambda_{i}\equiv 0 \pmod 2$ and $\lambda_{j}\equiv 1\pmod 2$, the length of $\Gamma $ is $[\lambda_{i}/2, \lambda_{j}]$, and there are $2(\lambda_{i}/2,\lambda_{j}) = 2(\lambda_{i},\lambda_{j})$ such $\Gamma$, each representing an associate pair. 
 \item Similarly, if $\lambda_{i}\equiv 1\pmod 2$ and $\lambda_{j}\equiv 1\pmod 2$, there are again $2(\lambda_{i},\lambda_{j})$ associate pairs.
 \item Lastly if $\lambda_{i}\equiv \lambda_{j}\equiv 1\pmod 2$, each $\pi$-cycle is also a $\pi^{2}$-cycle, so there are $(\lambda_{i},\lambda_{j})$ associate pairs represented.
 \end{itemize}
Thus each of these four cases contributes $(1+e_{ij})(\lambda_{i},\lambda_{j})$ to a count of associate pairs,  
and summing over $j \in \{i+1, \dots, r\}$ gives $\sum_{j>i}(1+e_{ij})(\lambda_{i},\lambda_{j})$ associate pairs from the non-square panels.  
Finally, summation over $i$ gives the result. 
\end{proof}

At this point, the following excursus is convenient.  Let $Y$ and $Y'$ be disjoint sets of cardinality $p$, bijective under the map $y\mapsto y'$, and say that a partial partition $\mathbf{P}$ on $Y\cup Y'$ is \emph{self-dual} if $\mathbf{P}=\mathbf{P}'$ (where $\mathbf{P}'$ is the partition on $Y'$ induced by the map $y\mapsto y'$).  

\begin{dfn}\label{combin}
Let $a_{p,q}$ be the number of self-dual (total) partitions of $Y\cup Y'$ of rank $2q$ which are orthogonal to the matching with edges $\{y,y'\}$, in the sense that 
$y,y'$ are never in the same block.  
\end{dfn}
Note that the rank of any partition as described in Definition \ref{combin} must be even, and that 
by definition, $a_{p,q}=0$ for $q>p$, $a_{p,0} = a_{0,q} = 0$ for all $p,q > 0$, and $a_{0,0} = 1$.  
\begin{lem} \label{seq}
For $p,q\geq 0$, $a_{p,q}$ is the \emph{diagonally scaled} Stirling number  $2^{p-q}\stirling{p}{q}$. 
\end{lem}
\begin{proof}
First we prove that the sequence (or array)  $a_{p,q}$ satisfies the recursion $$a_{p,q} = a_{p-1,q-1} + 2q\cdot a_{p-1,q} .$$  The equation holds for $p=1$, so let $p\ge 2$.  Given a set $Y=\{1,\dots, p-1\}$ and a partition $\mathbf{P} = \{P_{1},\dots, P_{2q}\}$ as described in Definition \ref{combin}, adjoin $p$ to $P_{j}$ and $p'$ to $P'_{j}$ to obtain a partition of $\{1,\dots, p\}$ meeting the conditions of Definition \ref{combin}.  This can be done in $2q$  distinct ways.  Alternatively, to a partition of $Y$ of rank $2q-2$ add new singleton blocks $\{p\}$ and $\{p'\}$ thus creating a self-dual partition of $\{1,\dots, p\}$ of rank $2q$ again meeting the conditions of Definition \ref{combin}.  All such self-dual partitions of $\{1,\dots, p\}$ of rank $2q$ may be constructed in one of these two ways, as may be seen by considering the effect of deleting elements $p, p'$.  Hence the recursion.  Secondly, it is simple to verify that $2^{p-q}\stirling{p}{q}$ also satisfies this recursion and  the same initial conditions as $a_{p,q}$. %
\end{proof}
This sequence is A075497 of the OEIS (\cite{OEIS}).  We use it to advance our cause:

\begin{pro} \label{pitau}
For given $\pi$ with cycle type $\lambda$, the number of semirigid fixed points of $\pi \tau$ acting on $\PP_{k}(X\times X)$ is 
 $$ F(\pi\tau) = \sum_{j=0}^{\zeta(\lambda)}\binom{\zeta(\lambda)}{j}\sum_{t=0}^{j} 
a_{j,t}\stirling{\zeta(\lambda)+\eta(\lambda) - j+1}{n-r-2t} .$$
\end{pro}

\begin{proof}
We use Theorems \ref{thm:C} and \ref{thm:sap} and their notation.  Choice of the domain of the partition $\mathbf{Q}$ may be made in $\binom{\zeta(\pi)}{j}$ ways, where $j\in \{0,\dots, \zeta(\lambda)\}$ is the size of the domain $Y$.  A partition of rank $2t\leq j$, constrained to be orthogonal to the matching, is placed on $Y$ in one of $a_{j,t}$ ways. There remain $\eta(\lambda)$ singular c-cycles and $\zeta(\lambda) - j$ associated pairs of regular c-cycles, and in order that $\mathbf{P}$ has rank $k$, on this set we place a partial partition $\mathbf{R}$ of rank $k-2t$.  This can be any of $\stirling{\zeta(\lambda)+\eta(\lambda) - j+1}{k-2t+1}$ partial partitions, and to conclude, $k+1=n-r$.
 \end{proof}
 
We now have a bound for the number of self-dual isomorphism classes.
\begin{thm}\label{thm:sd}
The number of self-dual isomorphism classes of semirigid 3-nilpotent semigroups of cardinality $n$ is bounded above by    
\begin{equation*}
 \sum_{r=1}^{n-2} \sum_{\lambda \vdash r} \frac{ 
\sum_{j=0}^{\zeta(\lambda)}\binom{\zeta(\lambda)}{j}\sum_{t=0}^{\min\{j,\floor{\frac{k}{2}}\}} 
a_{j,t}\stirling{\zeta(\lambda)+\eta(\lambda) - j+1}{n-r-2t} 
}
{\lambda_1^{\mu_1}\dots\lambda_t^{\mu_t}\mu_1!\dots \mu_t!} .
\end{equation*}
This is also a lower bound for the number of self-dual isomorphism classes of all  3-nilpotent semigroups of cardinality $n$.
\end{thm}

\begin{table}[bh]
\begin{center}
 \footnotesize{
 \begin{tabular}{l|r|r|r|r|r}
  $n$        & $a_n =$ \# semirigid  & $b_n =$
  Theorem \ref{thm:sd}& $c_n = $\# 3-nilp. self-dual 
  \\ \hline
  3  & 1      & 1                                        & 1               \\
  4  & 7      & 7                                        & 7              \\
  5  & 48     & 50                         & 50             \\
  6  & 639    & 649                     & 649           \\
  7           &19 475                   & 19 603
              &  19 605 \\%
 8           & -                          & 1851 244
              & -  \\%
  9           & -                          & 606 097 404
              & -  \\%
  10          & -                          &  608 877 118 483    & -

 \end{tabular}
 }
 \end{center}
\vspace{2mm}
\caption{The number of \emph{self-dual} 3-nilpotent semigroups up to isomorphism: semirigid actual, its upper bound from Theorem \ref{thm:sd}, and the actual total from Smallsemi \cite{GAP}.} 
\end{table}

\begin{proof}
By the orbit-counting lemma, 
the relevant number of orbits is bounded above by $\frac{\sum_{\pi\in\mathbb{S}_{r}}F(\pi\tau)}{r!}$.  Using the expression from Proposition \ref{pitau}, and grouping 
 the terms of this sum by rank $r$ and conjugacy class determined by $\lambda = (\lambda_1^{\mu_1},\dots,\lambda_t^{\mu_t})$, we have 
 \begin{equation*}
  \sum_{r=1}^{n-2} \sum_{\lambda \vdash r} \frac{ 
\sum_{j=0}^{\zeta(\lambda)}\binom{\zeta(\lambda)}{j}\sum_{t=0}^{j} 
a_{j,t}\stirling{\zeta(\lambda)+\eta(\lambda) - j+1}{n-r-2t} 
}
{\lambda_1^{\mu_1}\dots\lambda_t^{\mu_t}\mu_1!\dots \mu_t!} 
\end{equation*}
for the number of orbits.  Lastly, the upper limit $j$ in the summation over $t$ may be replaced by $\min\{j,\frac{k}{2}\}$ since $k-2t<0$ implies $\stirling{\cdots}{k-2t+1} = 0$.  
\end{proof}

We can now bound the classes up to equivalence (isomorphism or anti-isomorphism).  

\begin{thm}\label{equiv}
The number of non-equivalent semirigid 3-nilpotent semigroups of cardinality $n$ is bounded above by
\[
\begin{aligned}
\sum_{r=1}^{n-2} \sum_{\lambda \vdash r}  \frac{
 \stirling{\beta(\lambda) +1}{n-r} +  
\sum_{j=0}^{\zeta(\lambda)}\binom{\zeta(\lambda)}{j}\sum_{t=0}^{\min\{j,\floor{\frac{k}{2}}\}} 
a_{j,t}\stirling{\zeta(\lambda)+\eta(\lambda) - j+1}{n-r-2t} 
 }
{2 \lambda_1^{\mu_1}\dots\lambda_t^{\mu_t}\mu_1!\dots \mu_t!} .
\end{aligned}
\]  
This expression is also a lower bound for the number of all 3-nilpotent semigroups of cardinality $n$ up to equivalence.  
\end{thm}
\begin{proof}
We use Proposition \ref{folk} with upper bounds for $\mathscr{X}$ and $\mathscr{Y}$ supplied by Theorem \ref{count} and Theorem \ref{thm:sd} respectively. 
\end{proof}

\begin{table}[hb]
\begin{center}
 \footnotesize{
 \begin{tabular}{l|r|r|r|r|r}
  $n$        & $a_n =$ \# semirigid  & $b_n =$ Theorem \ref{equiv}& $c_n = $\# 3-nilp.  \\
    \hline
  3  & 1      & 1                                        & 1               \\
  4  & 8      & 8                                        & 8               \\
  5  & 81     &   83                         & 84              \\
  6  & 2634   &  2649                      & 2660            \\
  7  & 609~117 &    609 487               & 609~797        \\ 
  8  & -		& 1831 664 272  	& -	\\
  9  &   -	&	52 966 234 599 879  	&  -	\\
  10	& -	&	12 417 282 092 156 404 233	& -	\\
 \end{tabular}
 }
 \end{center}
\vspace{2mm}
\caption{The number of 3-nilpotent semigroups up to \emph{equivalence}: semirigid actual, the upper bound from Theorem \ref{equiv}, and the actual total from Smallsemi \cite{GAP}.}
\end{table}

\begin{exa}{1 (a final return)} \label{ex1-final}
Let us use Theorem \ref{thm:sd} to estimate the number of self-dual classes in the case $n=4$, which we know from the list of Example \ref{ex1} to be $7$. Here $r=2$ and $k=1$, and $ \lambda \in \{1^{2}, 2^{1}\}$. For both $\lambda = 1^{2}$ and $\lambda = 2^{1}$, $\eta=2$ and $\zeta=1$, so $j=0$ or $1$ and each of the relevant terms is $(\binom{1}{0}a_{0,0}\stirling{4}{2} + \binom{1}{1}a_{1,1}\stirling{3}{0})/2! = \frac{7}{2}$. 
Hence the sum evaluates to $7$. 
\end{exa}

\begin{exa}{7} \label{ex7}
To apply the  formula of Theorem \ref{equiv} for $n=5$, the cases are $r=2, k=2$ and $r=3, k=1$.  
With $r=2$, $\lambda = 1^{2}$ or $2^{1}$. As above, $\eta(1^{2})=\eta(2^{1})=2, \zeta(1^{2})=\zeta(2^{1})=1$, so $j\in\{0,1\}$ and the sum over $j$ is $\binom{1}{0}a_{00}\stirling{1+2-0+1}{2-0+1} + \binom{1}{1}a_{11}\stirling{1+2-1+1}{2-2+1} =\stirling{4}{3} + \stirling{3}{1} = 7$ for each $\lambda$; the denominator is $2$ so the contribution to the sum is $7$.  
With $r=3$, we have $\lambda=1^{3}, 1^{1}2^{1}$, or $3^{1}$, for which similar calculations give respective (non-zero) contributions of $\binom{3}{0}a_{0,0}\stirling{7}{2}/3! = 10.5$, $\binom{3}{0}a_{0,0}\stirling{7}{2}/2! = 31.5$, and $\binom{1}{0}a_{0,0}\stirling{3}{2}/3 = 1$.  All up this gives $50$ equivalence classes of self-dual \emph{semirigid} 3-nilpotent semigroups, and a value of $(50 + 116)/2 = 83$ in Theorem \ref{equiv}.   In comparison, the exact count for $n=5$ as computed by GAP is $84$ equivalence classes of \emph{all} 3-nilpotent semigroups, of which $50$ are self-dual. 
\end{exa}

\begin{exa}{8}\label{ex8}  
For $n=6$, we have $r\in \{2,3,4\}$ with (to cut it short) respective contributions to $S$ of $4$, $210$ and $435$, thus a total of $649$ self-dual classes, agreeing with Smallsemi \cite{GAP}.  The bound on  equivalence classes by Theorem \ref{equiv} is $2649$.
\end{exa}


\date{today} 
\end{document}